\documentclass[10pt]{article}
\usepackage{graphicx} 

\title{On the eigenvalue distribution of spatio-spectral limiting operators in higher dimensions, II}
\author{Kevin Hughes, Arie Israel, and  Azita Mayeli}
\date{\today}

\usepackage[left=2.5cm,right=2.5cm]{geometry}

\usepackage[backend=biber, style=alphabetic]{biblatex}

\usepackage[hidelinks]{hyperref}
 
\usepackage{amsmath, amsthm, amssymb, latexsym,epsfig,
amsthm,enumerate,multicol,wasysym}
\usepackage{color}
\usepackage{graphicx}
\usepackage{nccrules}
\usepackage{xfrac}

\usepackage{textcomp}
\usepackage{tikz}
\usepackage{xcolor}
\usepackage{comment}

\numberwithin{equation}{section}

\newcommand\blfootnote[1]{%
  \begingroup
  \renewcommand\thefootnote{}\footnote{#1}%
  \addtocounter{footnote}{-1}%
  \endgroup
}

\theoremstyle{plain}
\newtheorem{theorem}{Theorem}[section]
\newtheorem{lemma}[theorem]{Lemma}

\newtheorem{proposition}[theorem]{Proposition}

\theoremstyle{definition}

\newtheorem{case[theorem]}{Case}

\theoremstyle{remark}
\newtheorem{remark}[theorem]{Remark}

\numberwithin{equation}{section}

\def\R{\Bbb R}

\usepackage[normalem]{ulem}

\newcommand{\Z}{\mathbb Z}
\newcommand{\N}{\mathbb N}

\newcommand{\cW}{\mathcal W}

\newcommand{\cH}{\mathcal H}

\newcommand{\cQ}{\mathcal Q}
\newcommand{\cS}{\mathcal S}

\newcommand{\tr}{\mathrm{tr}}
\newcommand{\op}{\mathrm{op}}

\def\R{\mathbb{R}}

\begin{document}

\maketitle

\begin{abstract}
Let $F$, $S$ be bounded measurable sets in $\R^d$. Let $P_F : L^2(\R^d) \rightarrow L^2(\R^d) $ be the orthogonal projection on the subspace of functions with compact support on $F$, and let $B_S : L^2(\R^d) \rightarrow L^2(\R^d)$ be the orthogonal projection  on the subspace of functions with Fourier transforms having compact support on $S$. In this paper, we derive improved distributional estimates on the eigenvalue sequence $1 \geq \lambda_1(F,S) \geq \lambda_2(F,S) \geq \cdots > 0$ of the \emph{spatio-spectral limiting operator} $B_S P_F B_S : L^2(\R^d) \rightarrow L^2(\R^d)$. The significance of such estimates lies in their diverse applications in medical imaging, signal processing, geophysics and astronomy.  

For suitable domains $F$ and $S$, we prove that
\[
\# \{ k : \lambda_k(F,S) > \epsilon \}  = (2 \pi)^{-d} |F| \cdot |S| + \mathrm{Err}(F,S,\epsilon) \quad \mbox{for any } \epsilon \in (0,1),
\]
where $|F| \cdot |S|$ represents the Lebesgue measure of the domain $F \times S \subset \R^d \times \R^d$, and the error term satisfies the following bound:
\begin{align}\notag
|\mathrm{Err}(F,S,\epsilon)| \leq  C_d \frac{\cH_{d-1}(\partial F)}{\kappa_{\partial F}} \frac{\cH_{d-1}(\partial S)}{\kappa_{\partial S}}  \biggl\{ &  \log\left(\cH_{d-1}(\partial F)  \cH_{d-1}(\partial S)\right)  \log(\min \{ \epsilon, 1-\epsilon\}^{-1})^d  \\
\notag
& + \log \left(\cH_{d-1}(\partial F)  \cH_{d-1}(\partial S)\right)^3 \log(\min \{ \epsilon, 1-\epsilon\}^{-1}) \biggr\}
\end{align}
where $\cH_{d-1}(\partial F)$ and $\cH_{d-1}(\partial S)$ denote the $(d-1)$-dimensional Hausdorff measures of the boundaries of $F$ and $S$, respectively, while $\kappa_{\partial F}$ and $\kappa_{\partial S} \in (0,1]$ are geometric constants related to an Ahlfors regularity condition on the domain boundaries. When $F$ and $S$ are Euclidean balls, we expect this estimate to be sharp up to logarithmic factors. This improves on recent work of Marceca-Romero-Speckbacher \cite{MaRoSp23} which showed that $|\mathrm{Err}(F,S,\epsilon)| \leq  C_{d,\alpha} \frac{\cH_{d-1}(\partial F)}{\kappa_{\partial F}} \frac{\cH_{d-1}(\partial S)}{\kappa_{\partial S}}  \log\left(\frac{\cH_{d-1}(\partial F)  \cH_{d-1}(\partial S)}{\kappa_{\partial F} \min \{ \epsilon, 1-\epsilon \}}\right)^{2d(1+\alpha)+1}$ for any $\alpha >0$. 

Our proof is based on the decomposition techniques developed in \cite{MaRoSp23}. The novelty of our approach is in the use of a two-stage dyadic decomposition with respect to both the spatial and frequency domains, and the application of the results in \cite{ArieAzita23} on the eigenvalues of spatio-spectral limiting operators associated to cubical domains. 
\end{abstract}

\blfootnote{A.~Israel was supported by the Air Force Office of Scientific Research, under award FA9550-19-1-0005.  A.~Mayeli was supported in part by AMS-Simons Research Enhancement Grant and the PSC-CUNY research grants.}


\section{Introduction}

Let $F, S\subset \R^d$ be Borel measurable sets of finite Lebesgue measure. Define the {\it spacelimiting operator} $P_F : L^2(\R^d) \rightarrow L^2(\R^d)$ and {\it bandlimiting operator} $B_S : L^2(\R^d) \rightarrow L^2(\R^d)$ by
$P_F[f] := \chi_F f$, and $ B_S[f] := \mathcal{F}^{-1}[ \chi_S \widehat{f}]$, where $\widehat{f} \in L^2(\R^d)$ denotes the Fourier transform of $f \in L^2(\R^d)$, $\mathcal{F}^{-1}:  L^2(\R^d) \rightarrow L^2(\R^d)$ is the inverse Fourier transform, and $\chi_A$ is the indicator function of a  set $A \subset \R^d$. These operators are orthogonal projections on $L^2(\R^d)$ and they act on a function by restricting it to a given region in the spatial or frequency domain. 
We say that $f \in L^2(\R^d)$ is $F$-spacelimited if $P_F f = f$, and $f$ is $S$-bandlimited if $B_S f = f$.

By forming an alternating product of the spacelimiting and bandlimiting operators, we obtain a self-adjoint operator $T_{F,S}:= B_{S} P_F B_{S}$ on $L^2(\R^d)$, which we refer to as a {\it spatio-spectral limiting operator} (SSLO) associated to the {\it spatio-spectral} region $F \times S \subset \R^d \times \R^d$.

One form of the uncertainty principle asserts that a function and its Fourier transform cannot simultaneously have compact support. (For other manifestations of the uncertainty principle, see Remark \ref{UP1} below.) 
Quantitative variants of the uncertainty principle seek to determine the extent to which a function can be concentrated on a set $F$ while having Fourier transform concentrated on another set $S$. More concretely, let us fix Borel sets $F, S \subset \R^d$ and consider the optimization problem:
\begin{equation}\label{opt1:problem}
\lambda_1(F,S) = \max_{f \in L^2(\R^d)} \{ \| P_F f \|_{L^2(\R^d)}^2 : f = B_S f, \| f \|_{L^2(\R^d)} = 1 \}.
\end{equation}

That is, we ask to identify an $L^2$-normalized $S$-bandlimited function $f$ having maximal $L^2$ norm on the domain $F$. Observe that we can write $\| P_F f \|_{L^2(\R^d)}^2 = \langle P_Ff,  f \rangle = \langle P_F B_S f, B_S f \rangle = \langle B_S P_F B_S f, f \rangle$ provided $B_S f = f$. Thus, by the variational characterization of eigenvalues (\cite{Bell1, read1980functional}), the quantity $\lambda_1(F,S)$  is the largest eigenvalue of the operator $T_{F,S}= B_SP_FB_S$.  
Notice that, any bound of the form $\lambda_1(F,S)\leq c<1 $ may be considered an expression of the uncertainty principle; see \cite{donoho1989uncertainty}.

In this paper, we investigate the behavior of the entire eigenvalue sequence $(\lambda_k(F,S))_{k \geq 1}$ of $T_{F,S}$. 
Here, we shall write $1 \geq \lambda_1(F,S) \geq \lambda_2(F,S) \geq \cdots > 0$ for the non-increasing arrangement of the positive eigenvalues of ${T}_{F,S}$ (counted with multiplicity).  
A fundamental property is established in  \cite{Landau67}: Using the representation of $T_{F,S}$ as an integral operator and applying Mercer's theorem, it is true that
\begin{equation}
\label{eqn:trace_identity} \mathrm{tr}(T_{F,S}) = \sum_k \lambda_k(F,S) 
 = (2 \pi)^{-d} | F | |S| .
\end{equation}
Here and in what follows, we write $|X|$ to denote the Lebesgue measure of a Borel set $X \subset \R^d$.

Several works have considered the concentration properties of the eigenvalues $\lambda_k(F,S)$ in the interval $[0,1]$ which refine the identity \eqref{eqn:trace_identity}. For this discussion, we fix a threshold $\epsilon \in (0,1/2)$. Provided that the size of space-frequency domain $|F||S|$ is sufficiently large depending on $\epsilon$, one expects (roughly speaking) that there will be approximately $(2 \pi)^{-d} |F| |S|$ many eigenvalues in $[1-\epsilon,1]$, relatively fewer eigenvalues in $(\epsilon, 1 - \epsilon)$, and the remaining eigenvalues in $(0,\epsilon]$ converge rapidly to zero. That is, one predicts that the eigenvalues $\lambda_k(F,S)$ will tend to accumulate near $0$ or near $1$, with relatively few eigenvalues of intermediate size. Accordingly, $T_{F,S}$ should be well-approximated by a finite-rank projection operator up to an error of smaller rank.
Following \cite{Landau75}, we make this prediction precise through an asymptotic rescaling of the frequency domain: We consider the distribution properties of the eigenvalues $\lambda_k(F,rS)$ of the operator $T_{F,rS} = B_{rS}P_FB_{rS}$, as $r \rightarrow \infty$, where $rS := \{r x : x \in S\}$. For $\epsilon \in (0,1)$, let $N_\epsilon(r)$ be the number of eigenvalues $\lambda_k(F,rS)$ that are larger than $\epsilon$. If $F,S \subset \R^d$ are measurable, then Landau \cite{Landau75} proves that
\begin{equation}
\label{landau-eig-clustering_eqn}
\lim_{r \rightarrow \infty} r^{-d} N_\epsilon(r)  = (2 \pi)^{-d} |F| |S|.
\end{equation}
{\it A key observation} is that the method of proof of \eqref{landau-eig-clustering_eqn} in \cite{Landau75} does not yield a quantitative rate of convergence as $r\to \infty$, and this motivates the results in our paper.

In the following, we write $\cH_{d-1}(E)$ for the $(d-1)$-dimensional Hausdorff measure of a Borel set $E \subset \R^d$.

\begin{theorem}\label{thm1}
Let $d \geq 2$ and let $F,S \subset \R^d$ be bounded measurable sets having maximally Ahlfors regular boundaries with regularity constants $\kappa_{\partial F}, \kappa_{\partial S}  > 0$, respectively (see \eqref{Ahlfors_cond} for the definition of Ahlfors regularity).  Consider the spatio-spectral limiting operator $T = T_{F,S} =  B_{S} P_{F} B_{S}$, and let $\lambda_1(T) \geq \lambda_2(T) \geq \cdots > 0$ be its sequence of nonzero eigenvalues. Let $M_\epsilon(T) := \# \{ n : \lambda_n(T) \in (\epsilon,1-\epsilon)\}$ for $\epsilon > 0$.

If $\cH_{d-1}(\partial F) \cdot \cH_{d-1}(\partial S) \geq 64^{d-1}$ then for any $\epsilon \in (0,1/2)$,
\begin{equation}\label{eqn:main1}
\begin{aligned}
M_\epsilon(T) \lesssim  \frac{\cH_{d-1}(\partial F)}{\kappa_{\partial F}} \frac{\cH_{d-1}(\partial S)}{\kappa_{\partial S}}  \biggl\{ & \log\left(\cH_{d-1}(\partial F) \cdot \cH_{d-1}(\partial S)\right)  \log(\epsilon^{-1})^d \\
& + \log \left(\cH_{d-1}(\partial F) \cdot \cH_{d-1}(\partial S)\right)^3 \log(\epsilon^{-1}) \biggr\}.
\end{aligned}
\end{equation}
\end{theorem}
\noindent Here and throughout, $X \lesssim Y$ indicates that $X \leq CY$ for a constant $C$ depending only on the dimension $d$.

Our proof of Theorem  \ref{thm1} relies on the decomposition techniques developed in \cite{MaRoSp23}. Specifically, we determine a suitable decomposition of an SSLO $T_{F,S}$ into components $T_{Q_F,Q_S}$ where $Q_F$ and $Q_S$ are cubes contained in (or located near the boundary of) the domains $F$ and $S$, respectively. The decomposition process will be performed in two steps, first, with respect to the spatial domain $F$, and second, with respect to the frequency domain $S$. The Ahlfors condition on the boundaries of the domains is used to estimate the number of component cubes $Q_F$ or $Q_S$ at a fixed scale; see Proposition \ref{whitney_truncated:prop}. Estimates on the component operators $T_{Q_F,Q_S}$ are given in Lemma \ref{lem3}. We shall explain how to pass from estimates on the component operators $T_{Q_F,Q_S}$ to  estimates on the operator $T_{F,S}$. In Lemma \ref{decomp:lem} we present the case of a finite decomposition of the spatial domain $F$, with $S$ fixed. 

By adapting the techniques in \cite{MaRoSp23}, we pass from estimate \eqref{eqn:main1} on $M_\epsilon(T)$ to an estimate on the full distribution function $N_\epsilon(T) := \# \{ n  : \lambda_n(T) > \epsilon \}$ as follows.

\begin{theorem}
\label{thm2} 
Let $T= T_{F,S}$ be a spatio-spectral limiting operator as in the setting of Theorem \ref{thm1}. Then for any $\epsilon \in (0,1)$,
\begin{equation}\label{eqn:main2}
\begin{aligned}
|N_\epsilon(T) - (2 \pi)^{-d} |F| | S|| \lesssim \frac{\cH_{d-1}(\partial F)}{\kappa_{\partial F}} \frac{\cH_{d-1}(\partial S)}{\kappa_{\partial S}} &  \biggl\{  \log\left(\cH_{d-1}(\partial F)  \cH_{d-1}(\partial S)\right)  \log(\min\{\epsilon,1-\epsilon\}^{-1})^d \\
& + \log \left(\cH_{d-1}(\partial F)  \cH_{d-1}(\partial S)\right)^3 \log(\min\{\epsilon,1-\epsilon\}^{-1}) \biggr\}.
\end{aligned}
\end{equation}
\end{theorem}

See Section \ref{sec:pfrem1} for the proof of Theorem \ref{thm2}. We remark that the constant $(2 \pi)^{-d}$ appearing in the bound \eqref{eqn:main2} is dependent on the normalization of the Fourier transform, which we define by $\widehat{f}(\xi) = \int_{\R^d} f(x) e^{- i x \cdot \xi} dx$. Here, we have adopted the same normalization as \cite{Landau67, Landau75} -- see also \eqref{eqn:trace_identity} and \eqref{landau-eig-clustering_eqn}.

\begin{remark}\label{UP1}
There is a rich literature addressing the quantitative variants of the uncertainty principle. In these, one looks to quantify the extent to which an $S$-bandlimited function can be concentrated on a given spatial domain $F$ (or vice versa). Various refinements of the uncertainty principle have been addressed in the literature and are connected to classical topics in mathematics. For example, the Donoho-Stark uncertainty principle \cite{donoho1989uncertainty} and Tao's uncertainty principle \cite{tao2005} are fundamental results with applications in signal recovery and compressive sensing (for applications, see e.g.,  \cite{candestao2006,iosevich2023uncertainty}). Dyatlov's fractal uncertainty principle \cite{dyatlov2019introduction}  finds applications in ergodic theory and quantum chaos, while Moitra's uncertainty principle \cite{moitra2015super} is related to the problem of super-resolution in signal processing.

Several important uncertainty principles can be interpreted as giving bounds on $\lambda_1(F,S)$, or on the higher eigenvalues of $T_{F,S}$. For example, the fractal uncertainty principle \cite{dyatlov2019introduction} asserts that $\lambda_1(F,S) \leq C h^\beta$ for some $\beta > 0$, when $F$ and $S$ satisfy an Ahlfors regularity condition at scale above $h \in (0,1)$ (including, importantly, the case where $F$ and $S$ are neighborhoods of fractal sets). The continuous-time version of the Donoho--Stark  uncertainty principle asserts that $\lambda_1(F,S) \leq c_d |F| |S|$, where $c_d$ is an  explicit constant depending on $d$ and determined by the normalization of the Fourier transform. ({In fact, the continuous-time version of the Donoho-Stark uncertainty principle is immediate from the trace identity \eqref{eqn:trace_identity}, since, for the normalization of the Fourier transform adopted here, $\lambda_1(F,S) \leq \tr(T_{F,S}) = (2\pi)^{-d} |F| |S|$.}) 
Moitra's uncertainty principle \cite{moitra2015super}  gives bounds on the largest and smallest eigenvalues of a finite-dimensional analogue of the operator $T_{F,S}$ acting on $\ell^2(\Z/N\Z)$, where the $F$ and $S$ correspond to subsets of $\Z/N \Z$ satisfying a separation condition. In fact, Moitra's uncertainty principle relates to the problem of determining the condition number of submatrices of Vandermonde matrices (see also \cite{Pan2016}).
\end{remark}

\subsection{Comparison to previous results}

In \cite{Karnik21}, a near-sharp upper bound is given for the eigenvalue counting function for one-dimensional spatio-spectral limiting operators associated to a pair of intervals: Let $W\geq 2\pi $, and let $\lambda_k(W)$ be the sequence of eigenvalues for the SSLO  $T_W = T_{[0,1],[-W,W]} = B_{[-W,W]}P_{[0,1]}B_{[-W,W]}$ on $L^2(\R)$. In  \cite{Karnik21}, the authors prove an upper bound   for the counting function $M_\epsilon(W):=\# \{ k : \lambda_k(W) \in (\epsilon, 1 - \epsilon) \}$ by 
\begin{align}\label{Romberg-Thm:3} 
M_\epsilon(W)\leq \frac{2}{\pi^2} \log \left( \frac{50W}{\pi}+25 \right) \log \left( \frac{5}{\epsilon(1-\epsilon)} \right) + 7.
 \end{align}
The dependence on $W$ in \eqref{Romberg-Thm:3} is asymptotically sharp for $\epsilon$ fixed and small, since Landau--Widom \cite{LandauWidom80} proved that $M_\epsilon(W) = \frac{2}{\pi^2} \log(W) \log\left( \frac{1}{\epsilon}-1 \right) + o(\log(W))$ as $W \rightarrow \infty$.

As noted in \cite{Karnik21}, the equation \eqref{Romberg-Thm:3} implies the exponential decay of $\lambda_k(W)$ for $k \geq \lceil W/\pi \rceil$: 
\begin{equation}\label{1d_eigdecay}
\lambda_{k}(W) \leq 10 \exp \left[ - \frac{k - \lceil W/\pi \rceil - 6}{c_1 \log(W + c_2)}  \right] \qquad k \geq \lceil W/\pi \rceil,
\end{equation}
where the constants $c_1, c_2 > 0$ are explicit, and $\lceil x \rceil$ is the ceiling function of $x$.

We remark that the eigenfunctions of the one-dimensional SSLOs associated to a pair of intervals are the well-known Prolate Spheroidal Wave Functions arising in mathematical physics \cite{Bell1, Bell2, Bell3}. We refer the reader to \cite{ Osipov13, Israel15, Karnik21, Bonami21} 
for additional results on the distribution of eigenvalues of SSLOs in the one-dimensional case. The discrete version of the problem has also been considered  \cite{Bell5, Karnik21}.

In higher dimensions, when the spatial and frequency domains are balls in $\R^d$, the eigenfunctions and eigenvalues of SSLOs were examined in \cite{Bell4, Shkolnisky07}. In this case, the eigenfunctions of SSLOs are also the eigenfunctions of a related differential operator. These functions form an orthogonal system in $L^2(\R^d)$, known as \emph{prolate functions}, and have been used for applications in scientific imaging problems, e.g., in cryoelectron microscopy (cryo-EM) \cite{Lederman17, Shkolnisky17A, Shkolnisky17B, LedermanSinger17, LedermanSinger20} and MRI   \cite{Yang02}. The papers \cite{Lederman17,GreengardSerkh18} develop numerical techniques for computing prolate functions and the corresponding eigenvalues, and describe applications in numerical analysis for the processing of bandlimited functions (e.g., methods for quadrature and interpolation).

To the best of our knowledge, the only results on quantitative bounds for the eigenvalues of SSLOs in the higher-dimensional case were established in \cite{ArieAzita23} and \cite{MaRoSp23}, which appeared as arXiv preprints almost simultaneously in 2023.
For example, in \cite{ArieAzita23}, the second and third-named authors proved the following theorem: 
\begin{theorem}\label{mainthm:cube_convex}
Let $F = [0,1]^d$ be the standard unit hypercube. Let $S \subset B(0,1)$ be a compact convex set in $\R^d$ that is closed under coordinate reflections; that is, if $ (x_1,\cdots,x_d) \in S$ then $(x_1,\cdots, - x_j, \cdots, x_d) \in S$ for any $j$. For $r > 0$, let $rS$ denote the $r$-dilate of $S$. Then 
for any $\epsilon \in (0,1/2)$ and $r \geq 1$, 
\begin{align}\label{eig_clust:eqn1}
&| \# \{ k : \lambda_k(F,rS) > \epsilon\} - (2\pi)^{-d} | rS | | \leq C_d E_d(\epsilon,r) ,\\
\label{eig_clust:eqn2}
& \# \{ k : \lambda_k(F,rS) \in (\epsilon, 1-\epsilon)\} \leq C_d E_d(\epsilon,r), \\
\label{eig_clust:eqn3}
&E_d(\epsilon,r) := \max \{ r^{d-1} \log(r/\epsilon)^{5/2}, \log(r/\epsilon)^{5d/2} \}.
\end{align}
Here,  $C_d > 0$  is  a dimensional constant and is independent of $S$. 
\end{theorem}

Bounds \eqref{eig_clust:eqn1} and \eqref{eig_clust:eqn2} describe the clustering behavior of the SSLO eigenvalues near $0$ and $1$. Observe that $(2\pi)^{-d} |r S| = (2\pi)^{-d} |S| r^d$, while the term on the right-hand side of \eqref{eig_clust:eqn1} and \eqref{eig_clust:eqn2} is of lower order $\sim e(r) := r^{d-1} \log(r)^{5/2}$ for fixed $\epsilon>0$ and large $r$. Thus, Theorem \ref{mainthm:cube_convex} implies that the first $(2\pi)^{-d} |S| r^d - O(e(r))$ many eigenvalues will be very close to $1$, followed by at most $O(e(r))$ many eigenvalues of intermediate size, and the remaining eigenvalues decay to $0$ at a near-exponential rate (see Corollary 1.5 in \cite{ArieAzita23}
 for a  statement about the decay rate). In particular, \eqref{eig_clust:eqn1} gives a quantitative version of Landau's formula \eqref{landau-eig-clustering_eqn}.

In  \cite{MaRoSp23}, the authors studied the distribution of the eigenvalues $\lambda_k(F,S)$ for a wider class of domains $F$ and $S$ in $\R^d$, such that the boundaries $\partial F$ and $\partial S$ are maximally Ahlfors regular with regularity constants $\kappa_{\partial F}$ and $\kappa_{\partial S}$, and their Hausdorff $(d-1)$-measures satisfy $\cH_{d-1}(\partial F) \cdot \cH_{d-1}(\partial S) \geq 1$. Under this assumption, the authors prove an estimate of the following form: For every $\alpha \in (0,1/2]$ there exists $A_{\alpha, d} \geq 1$ such that for all $\epsilon \in (0,1/2)$,
\begin{equation}\label{Romero_bd}
\# \{k \in \N : \lambda_k(F,S) \in (\epsilon,1-\epsilon)\} \leq A_{\alpha,d} \frac{\cH_{d-1}(\partial F)}{\kappa_{\partial F}} \frac{\cH_{d-1}(\partial S)}{\kappa_{\partial S}} \log \left( \frac{\cH_{d-1}(\partial F) \cH_{d-1}(\partial S)}{\kappa_{\partial F} \; \epsilon}\right)^{2d(1+\alpha) + 1}.
\end{equation}

In this paper, we appeal to the domain decomposition techniques of \cite{MaRoSp23} and strictly sharpen the bound \eqref{Romero_bd} for the same class of domains. Indeed, comparing \eqref{Romero_bd} with our bound \eqref{eqn:main1}, we see that the exponent of the $\log(\epsilon^{-1})$ term is reduced from $2d(1+\alpha) + 1$ to $d$, and the exponent of the $\log(\cH_{d-1}(\partial F) \cH_{d-1}(\partial S))$ term is reduced from $2d(1+\alpha) + 1$ to $3$. 

When $F = [0,1]^d$ and $S$ is a domain with maximally Ahlfors regular boundary, the estimate \eqref{eqn:main1} implies that
\[
\#\{k : \lambda_k(F,rS) \in (\epsilon,1-\epsilon)\} \leq C r^{d-1} \left\{ \log(r) \log(\epsilon^{-1})^d + \log(r)^3 \log(\epsilon^{-1})\right\},
\]
where $C$ is a constant determined by $S$ and $r$ is sufficiently large. In fact, a single factor of $\log(r)$ can be removed from the above right-hand side by appealing to Theorem \ref{stageIdecomp:thm}. However, if $F=[0,1]^d$ and $S$ is convex and symmetric with respect to coordinate reflections, the bound \eqref{eig_clust:eqn2} is strictly stronger in the regime when $d \geq 3$, $\epsilon$ is proportional to a negative power of $r$, and $r \rightarrow \infty$. Therefore, the results of \cite{ArieAzita23} yield an improvement over \eqref{eqn:main1} and \eqref{Romero_bd} for a narrower class of domains and a particular range of parameters.

\subsection{Organization of the paper}  The rest of the article is organized as follows. After stating the preliminaries and establishing notation in Section \ref{Prelim}, we develop techniques for the dyadic decomposition of Ahlfors regular domains in Section \ref{dyadic approximation of domains}, a key aspect of the study. In Section \ref{framework}, we discuss the class of $p$-Schatten operators, outlining their significance in the context of the paper. We also describe the basic properties of spatio-spectral limiting operators.  We conclude this section with a key result on the behaviour of the eigenvalues for a sum of spatio-spectral limiting operators. In
 Section  \ref{sec:dd1}, we apply the decomposition technique to the spatial domain. We conclude the proof of Theorem \ref{thm1} in Section  \ref{sec:dd2} by performing a decomposition of the frequency domain. Section \ref{sec:pfrem1} is dedicated to giving a proof of Theorem \ref{thm2}.


\section{Notation}\label{Prelim}

For ease of future reference, we collect here some of the notation that will be used throughout.

\begin{itemize}
\item Unless stated otherwise, we will assume that \(\epsilon \in (0,1/2)\). 

\item Write $\log(\cdot)$ for the base $2$ logarithm and $\ln(\cdot)$ for the natural logarithm. Write $d$ for the ambient dimension of the Euclidean space, e.g., $\R^d$.  The Euclidean norm on $\R^d$ is $\| x \| := \left(\sum_i x_i^2 \right)^{1/2}$ for $x \in \R^d$. Let $B(x,r) = \{ z \in \R^d : \| z-x \| \leq r\}$ be the closed ball with center $x\in\R^d$ and radius $r$. 

For a measurable subset $E \subset \R^d$, let $|E|$ be the Lebesgue measure of $E$, and let  $\cH_{d-1}(E)$ be the $(d-1)$-dimensional Hausdorff measure of $E$.

\item For $t > 0$ and $\Omega \subset \R^d$, let $t \Omega := \{ tx : x \in \Omega\}$, where $tx= (tx_1, \cdots, tx_d)$.

\item A \emph{cube} is a set of the form $Q = c + \left[-\frac{\delta}{2}, \frac{\delta}{2}\right]^d$, with \emph{center} $c_Q := c \in \R^d$ and \emph{sidelength} $\delta_Q := \delta$.  

\item Write $\chi_A$ for the indicator function of a set $A \subset \R^n$. Thus, $\chi_A(x) = 1$ if $x \in A$, and $\chi_A(x) = 0$ if $x \in \R^n \setminus A$. 

\item Write $\widehat{g}$ to denote the Fourier transform of $g \in L^2(\R^d)$, defined by $\widehat{g}(\xi) = \int_{\R^d} g(x) e^{-ix \cdot \xi} dx$.
The inverse Fourier transform $\mathcal{F}^{-1} $ on $L^2(\R^d)$ is then defined by  $\mathcal{F}^{-1}[f](x) = \frac{1}{(2 \pi)^d} \int_{\R^d} f(\xi) e^{i x \cdot \xi} d \xi$.

\item For $X,Y \geq 0$ we use the notation $X \lesssim Y$ to indicate that there exists a constant $C$, depending only on the dimension $d$ -- but not on $X,Y$ -- such that $X \leq C Y$.  We write $X \sim Y$ to denote that $X \lesssim Y$ and $Y \lesssim X$.
\end{itemize}

\section{Dyadic approximations of Ahlfors regular domains}\label{dyadic approximation of domains}

Following \cite{MaRoSp23}, we say that a set $\Omega \subset \R^d$ ($d \geq 2$) has \emph{maximally Ahlfors regular} boundary with \emph{regularity constant} $\kappa_{\partial \Omega}>0$ provided that
\begin{equation}
    \label{Ahlfors_cond}
    \cH_{d-1}(\partial \Omega \cap B(x,r)) \geq \kappa_{\partial \Omega} r^{d-1} \qquad \forall \ x \in \partial \Omega, \; 0 < r \leq \cH_{d-1}(\partial \Omega)^{1/(d-1)}.
\end{equation} 
Here, the term `maximal' pertains to the range of $r$ for which the estimate in \eqref{Ahlfors_cond} is valid.  We note that $\kappa_{\partial \Omega}\in (0,1]$, as follows by taking $r= \cH_{d-1}(\partial \Omega)^{1/(d-1)}$ in \eqref{Ahlfors_cond}.

For $t > 0$ and $\Omega \subset \R^d$, by the scaling properties of the Hausdorff measure we have
\begin{equation}\label{eqn:scaling1}
\cH_{d-1}(\partial t \Omega) = t^{d-1} \cH_{d-1}(\partial \Omega).
\end{equation}
Further, if $\Omega$ has maximally Ahlfors regular boundary with regularity constant $\kappa_{\partial \Omega}$, then $t \Omega$ has maximally Ahlfors regular boundary with regularity constant $\kappa_{\partial \Omega}$. That is,
\begin{equation}\label{eqn:scaling2}
\kappa_{\partial t \Omega} = \kappa_{\partial \Omega}.
\end{equation}
To prove \eqref{eqn:scaling2}, we simply apply \eqref{eqn:scaling1} in the definition \eqref{Ahlfors_cond}.

We shall use these scaling properties in Sections \ref{sec:dd1} and \ref{sec:dd2} of the paper.

The following proposition elaborates on a cubical covering technique used in the proof of Lemma~5.2 of \cite{MaRoSp23}. We isolate the properties of the covering for future use in Sections \ref{sec:dd1} and \ref{sec:dd2}, and include the proof for completeness.
\begin{proposition}\label{whitney_truncated:prop}
Assume that $\Omega \subset \R^d$ is a bounded set having a maximally Ahlfors regular boundary with  regularity constant $\kappa_{\partial \Omega} > 0$. Take $0 < \eta \leq \frac{1}{2} \cH_{d-1}(\partial \Omega)^{1/(d-1)}$. There exists a finite family $\cW = \cW(\eta)$ of cubes with pairwise disjoint interiors and   sidelengths being integer powers of $2$, satisfying the following  properties: 
\begin{enumerate}
\item The set $\cW$ is a cover for  $\Omega$, i.e., $\Omega \subset \bigcup_{Q \in \cW} Q$.
\item If $Q \in \cW$ with $\delta_Q \geq 2 \eta$, then $Q \subset \Omega$.
\item If $Q \in \cW$, then $\eta \leq \delta_Q \leq \cH_{d-1}(\partial\Omega)^{1/(d-1)}$.
\item For every $\delta \geq \eta$, 
\[
\# \{ Q \in \cW : \delta_Q = \delta \} \leq 6^d \frac{\cH_{d-1}(\partial \Omega)}{\kappa_{\partial \Omega}} \delta^{1-d}.
\]
\end{enumerate}
\end{proposition}
\begin{proof}
For $\alpha > 0$ and a cube $Q$ with center $c_Q$ and sidelength $\delta_Q$,   we denote by $\alpha Q$ the cube with center $c_Q$ and sidelength $\alpha \delta_Q$. In other words, $\alpha Q$ is obtained from $Q$ via concentric scaling by a factor of $\alpha$.

We refer to a cube $Q \subset \R^d$ as \emph{dyadic} if $\delta_Q = 2^k$ and $c_Q \in 2^k \Z^d$ for some integer $k \in \Z$.

Because $\Omega$ is bounded, by applying a translation, we may assume that $\Omega$ is contained in the orthant $[0,\infty)^d$, and so there exists a dyadic cube $Q^0$ satisfying $\Omega \subsetneq Q^0$ and $\delta_{Q^0} > 2 \eta$. 

We say a dyadic cube $Q \subset Q^0$ is OK provided that either (a) $\delta_Q < 2 \eta$ and $Q \cap \Omega \neq \emptyset$, or (b) $Q \subset \Omega$. Let $\cW$ be the set of \emph{maximal} dyadic cubes $Q$ that are OK -- that is, $Q$ is OK, but no dyadic cube strictly containing $Q$ is OK. Because two dyadic cubes are either nested or have disjoint interiors, the cubes in $\cW$ have pairwise disjoint interiors.

\noindent{\it Proof of 1:} Each  $x \in \Omega$ is in a dyadic cube satisfying (a), therefore each $x \in \Omega$ is in an OK cube. So, $\cW$ covers $\Omega$, proving condition 1.

\noindent{\it Proof of 2:} 
If $Q \in \cW$ with $\delta_Q \geq 2 \eta$ then $Q$ cannot satisfy (a), so $Q$ must satisfy (b) -- thus, $Q \subset \Omega$. This proves condition 2. 

\noindent{\it Proof of 3:} Note that $\delta_Q \geq \eta$ for $Q \in \cW$ -- if not, $\delta_Q <\eta$, and the dyadic parent $Q^+ \subset Q^0$ of $Q$ satisfies (a), contradicting the maximality of $Q$. To prove $\delta_Q \leq \cH_{d-1}(\partial \Omega)^{1/(d-1)}$, notice that this happens if  $\delta_Q \leq 2 \eta \leq \cH_{d-1}(\partial \Omega)^{1/(d-1)}$. So we assume that $\delta_Q > 2 \eta$. By condition 2, $Q  \subset \Omega$. Let $\pi : \R^d \rightarrow \R^d$ denote the orthogonal projection onto the hyperplane $\{x_1 = 0\}$. Then $\pi(Q) \subset \pi(\Omega) = \pi(\partial \Omega)$. Using that $\cH_{d-1}(\pi(\partial \Omega)) \leq \cH_{d-1}(\partial \Omega)$, and $\cH_{d-1}(\pi(Q)) = \delta_Q^{d-1}$, we complete the proof of condition 3. 

\noindent{\it Proof of 4:} We first establish that
\begin{equation}\label{eqn:nearbd}
3Q \cap \partial \Omega \neq \emptyset \qquad \mbox{ for each } Q \in \cW.
\end{equation}
To prove \eqref{eqn:nearbd}, let $Q \in \cW$. Whether $Q$ satisfies (a) or (b), it holds that $3Q \cap \Omega \neq \emptyset$. We claim  that $3Q$ contains a point of $\R^d \setminus \Omega$. If not, then $3Q \subset \Omega$, and since $Q^+ \subset 3Q$, the cube $Q^+$ satisfies property (b), contradicting the maximality of $Q$. 
Since $3Q$ contains both points of $\Omega$ and $\R^d \setminus \Omega$,  it follows that $3Q \cap \partial \Omega \neq \emptyset$,  establishing \eqref{eqn:nearbd}.

For $Q \in \cW$,  put $S_Q:= 5Q \cap \partial \Omega$.
We will prove  two bounds which quickly yield condition 4. 
First, we have a lower bound on the Hausdorff measure of  $S_Q$:
\begin{equation}\label{bound:Hausdorff_measure}
\cH_{d-1}(S_Q) \geq  \kappa_{\partial \Omega} \delta_Q^{d-1}
\quad \text{for each  } Q \in \cW.
\end{equation}
Next, we have that, for fixed $\delta > 0$, the family $\{ S_Q : Q \in \cW, \; \delta_Q = \delta\}$ of subsets of $\partial \Omega$ has covering number at most $6^d$ in the sense that
\begin{equation}\label{bound:covering}
\#\{Q \in \cW: x \in S_Q, \; \delta_Q = \delta \} \leq 6^d
\quad \text{for each  } x \in \partial \Omega.
\end{equation}
By summing both sides of \eqref{bound:Hausdorff_measure}  over all  $Q \in \cW$ with $\delta_Q = \delta$, and then using  the bound \eqref{bound:covering} and the additivity of the Hausdorff measure, we obtain 
\[
\# \{ Q \in \cW : \delta_Q = \delta \} \cdot \kappa_{\partial \Omega} \delta^{d-1} \leq \sum_{Q \in \cW, \; \delta_Q = \delta} \cH_{d-1}(S_Q) \leq 6^d \cH_{d-1}(\partial \Omega).
\]
This completes the proof of condition 4 subject to proving the bounds  \eqref{bound:Hausdorff_measure} and \eqref{bound:covering}, which we now treat in succession. 

We write $\|x \|_\infty = \max_i |x_i|$ for the $\ell^\infty$ (maximum) norm of a vector $x \in \R^d$. Observe that if $Q$ is a cube then $x \in Q$ if and only if $\| x - c_Q \|_\infty \leq \delta_Q/2$.

To prove \eqref{bound:Hausdorff_measure}, fix $Q \in \cW$.
By \eqref{eqn:nearbd}, there is a point $x \in 3Q \cap \partial \Omega$. We claim that $B(x,\delta_Q) \subset 5Q$. To see this let $y \in B(x,\delta_Q)$, and note that 
\[
\| y - c_Q \|_\infty \leq \| y - x \|_\infty + \| x - c_Q \|_\infty \leq \| y - x \| + \| x - c_Q \|_\infty \leq \delta_Q + 3 \delta_Q/2 = 5 \delta_Q/2,
\]
hence, $y \in 5Q$, proving the desired containment. By the Ahlfors regularity condition \eqref{Ahlfors_cond}, and since $\delta_Q \leq \cH_{d-1}(\partial\Omega)^{1/(d-1)}$ by condition 3, we then have 
\[
\cH_{d-1}(S_Q) = \cH_{d-1}(\partial \Omega \cap 5Q) \geq \cH_{d-1}(\partial \Omega \cap B(x, \delta_Q)) \geq \kappa_{\partial \Omega} \delta_Q^{d-1}.
\]
This completes the proof of \eqref{bound:Hausdorff_measure}. 

To prove \eqref{bound:covering}, fix $x \in \partial \Omega$ and $\delta > 0$, and let $Q_x$ be the closed cube centered at $x$ of sidelength $6 \delta$. If $x \in 5Q$ for some $Q \in \cW$ with $\delta_Q = \delta$, then  $Q \subset Q_x$. 
To see this, let $y \in Q$ and estimate $\|x-y\|_\infty \leq \| x - c_Q \|_\infty + \| c_Q - y \|_\infty \leq 5 \delta_Q/2 + \delta_Q/2 = 3 \delta$, which implies that $y \in Q_x$. 
Returning to \eqref{bound:covering}, let $Q \in \cW$ satisfy $x \in S_Q$ and $\delta_Q = \delta$. The above tells us $Q \subset Q_x$. Since the volume of $Q_x$ is $(6\delta)^d$, the cubes of $\cW$ have pairwise disjoint interiors, and each $Q$ has volume $\delta^d$, a comparison of volumes shows that there are at most $6^d$ cubes $Q \in \cW$ satisfying $x \in S_Q$ and $\delta_Q = \delta$.
This completes the proof of \eqref{bound:covering} and also the proof of condition 4.  
\end{proof}

\section{Analytical framework:  Schatten operators, spatio-spectral limiting  operators, and decomposition techniques}\label{framework}

\subsection{Preliminaries on Schatten operators}

Write $\| T \|_p$ for the Schatten $p$-norm of a bounded operator $T : \cH \rightarrow \cH$ on a complex Hilbert space $\cH$, given by $\| T \|_p = \mathrm{tr}( (T^*T)^{p/2})^{1/p}$ for $p \in (0,\infty)$, and $\| T \|_\infty = \| T \|_{\op}$ for the operator norm of $T$. Observe that $\| T \|_{1} = \tr(|T|)$. 
If $\| T \|_p$ is finite, we say $T$ is a $p$-Schatten operator and it belongs to  the  Schatten $p$-class $\cS_p$.  The spaces $\cS_1$ and $\cS_2$ consist of the trace class operators and Hilbert-Schmidt operators on $\cH$,  respectively. 
In a slight abuse of notation, we use the term ``norm" to refer to $\|\cdot\|_p$ for $p \in (0,1)$, even though it is not subadditive. However, for $0 < p \leq 1$, the $p$'th power of the Schatten $p$-norm is subadditive \cite{McCarthy67}, i.e.
\begin{equation}\label{schatten_subadditive:eqn}
\| T_1 + T_2 \|_p^p \leq \| T_1 \|_p^p + \| T_2 \|_p^p, 
\end{equation}
for any operators $T_1$ and $T_2$ in $\cS_p$.

If $T$ is a compact, positive semidefinite  (self-adjoint) operator on a complex Hilbert space, we denote the positive spectrum   of $T$ by $\sigma(T) := ( \lambda_n(T) )_{n \geq 1}$ where $\lambda_1(T) \geq \lambda_2(T) \geq \lambda_3(T) \geq \cdots > 0$ are the positive eigenvalues of $T$ counted with multiplicity. Observe that 
for finite $p\in (0,\infty)$, the Schatten $p$-norm  of $T$ admits the representation, 
\[
\| T \|_p = \left( \sum_{n=1}^\infty \lambda_n(T)^p \right)^{1/p}, 
\]
and,  $\|T\|_\infty = \lambda_1(T)$.

For any compact operator $0 \leq T \leq 1$ and $\epsilon \in (0,1/2)$, we define the \emph{eigenvalue counting function} of $T$ as  
\begin{equation}\label{counting-function}
M_\epsilon(T):=\# \{ n : \lambda_n(T) \in (\epsilon,1-\epsilon)\}.
\end{equation}
Define \( g_{\epsilon,p} : [0,1] \rightarrow \mathbb{R} \) by \( g_{\epsilon,p}(t) = (\epsilon - \epsilon^2)^{-p} (t - t^2)^{p} \). Observe that \( g_{\epsilon,p}(t) \geq 0 \) for \( t \in [0,1] \), and we have \( g_{\epsilon,p}(t) \geq 1 \) for \( t \in (\epsilon, 1-\epsilon) \). Hence, \( g_{\epsilon,p}(t) \geq \chi_{(\epsilon,1-\epsilon)}(t) \) for \( t \in [0,1] \).
Therefore, for any $p > 0$,
\begin{equation}\label{cheb:ineq}  
M_\epsilon(T)  =  \text{tr}(\chi_{(\epsilon, 1-\epsilon)} (T)) \leq \text{tr}(g_{\epsilon,p}(T)) 
=
(\epsilon - \epsilon^2)^{-p} \| T - T^2 \|_p^p \leq 2^p \epsilon^{-p} \| T - T^2 \|_p^p.
\end{equation}

According to \eqref{cheb:ineq}, a bound on the Schatten norm $\| T - T^2 \|_p^p$ implies a bound on $M_\epsilon(T)$. The following result is adapted from  Lemma 5.1 of \cite{MaRoSp23} and gives a partial converse to \eqref{cheb:ineq}.
\begin{lemma}\label{lem1}
Suppose that for a compact operator $0 \leq T \leq 1$, there are constants $a, A_1, A_2 > 0$ such that for every $\epsilon \in (0,1/2)$, 
\[
M_\epsilon(T) \leq A_1 \ln(\epsilon^{-1})^a  + A_2 \ln(\epsilon^{-1}).
\]
Then for every $0 < p \leq 1$,
\[
\| T - T^2 \|_p^p \leq \Gamma(a+1) A_1 p^{-a} + A_2 p^{-1}.
\]
\end{lemma}
\begin{proof}
Observe that 
\begin{align}\label{inequality}
(x-x^2)^p \leq  p\int_0^{1/2} \chi_{(t,1-t)}(x)  t^{p-1} dt \qquad \mbox{for } 0 \leq x \leq 1.
\end{align}
To prove \eqref{inequality}, it suffices to assume $0 \leq x \leq 1/2$ by symmetry of  the  functions $x\to (x-x^2)^p$  and 
$x\to  \int_0^{1/2} \chi_{(t,1-t)}(x) p t^{p-1} dt$
about $x=1/2$; note $\chi_{(t,1-t)}(x) = \chi_{(0, x)}(t)$ if $t , x \in [0,1/2]$ and so the right-hand side of \eqref{inequality} evaluates to $ x^p \geq x^p(1-x)^p = (x-x^2)^p $.

Since $\mathrm{tr}(\chi_{(t,1-t)}(T))=M_t(T)$ for $t \in (0,1/2)$, from \eqref{inequality}  we deduce that
\begin{align*}
\| T - T^2 \|_p^p \leq \int_0^{1/2} M_t(T) p t^{p-1} dt &\leq  \int_0^{1} \left( A_1 \ln(t^{-1})^a + A_2 \ln(t^{-1}) \right)  p t^{p-1} dt \\
&= \int_0^{\infty} \left( A_1 (u/p)^a + A_2 (u/p) \right)  e^{-u} du \\
&=  \Gamma(a+1) A_1 p^{-a} + A_2 p^{-1},
\end{align*}
using the change of variables $u = p\ln(1/t)$, so that $t = e^{-u/p}$ and $dt = - (1/p) e^{-u/p} du$.
\end{proof}

\subsection{Preliminaries on spatio-spectral limiting operators}
In tandem with the SSLO \(T_{F,S} = B_S P_F B_S\), the corresponding \emph{Hankel operator} is defined by
\[ H_{F,S} := (1-B_S) P_F B_S : L^2(\R^d) \rightarrow L^2(\R^d). \] 
It is clear that the following identity is satisfied: 
\begin{equation}\notag 
T_{F,S} - T_{F,S}^2 = H_{F,S}^*  H_{F,S}.
\end{equation}
Thus, for any $p > 0$,
\begin{equation}\label{op_fact2:eqn}
\| T_{F,S} - T_{F,S}^2 \|_p^p = \| H_{F,S} \|_{2p}^{2p}.
\end{equation}

We will make use of the following trace and spectral identities in Sections \ref{sec:dd1} and \ref{sec:dd2}. For the elementary proofs of these identities, see Lemma 1 of \cite{Landau67}.

\begin{proposition}
\label{prop:SSLO_props} 
For measurable sets $F, S \subset \mathbb{R}^d$ of finite Lebesgue measure, $\tr(T_{F,S}) = (2\pi)^{-d}  |F| \cdot |S|$. Furthermore, the ordered sequence of positive eigenvalues of $T_{S,F} = B_F P_S B_F$ and $T_{F,S} =B_SP_F B_S$ are identical, i.e., $\lambda_k(S,F) = \lambda_k(F,S)$ for all $k$.  
\end{proposition}


To estimate the distribution of the eigenvalues of an SSLO $T_{F,S}$ we decompose it into components $T_{Q_F,Q_S}$, where $Q_F$ and $Q_S$  belong to families of component cubes contained in (or located near the boundary of) the domains $F$ and $S$, respectively. We will later explain how to pass from estimates on the component operators $T_{Q_F,Q_S}$ to estimates on $T_{F,S}$ (see Lemma \ref{decomp:lem} for the case of finite decompositions of the spatial domain $F$, with $S$ fixed).  Estimates on the component operators are addressed in the following lemma, which is stated equivalently as Theorem 1.2 in \cite{ArieAzita23}; the proof of this result relies on the results of \cite{Karnik21} for one-dimensional SSLOs. 
\begin{lemma}[Bounds on SSLO eigenvalues associated to spatial and frequency cubes]
\label{lem3}
Let $Q_1$ and $Q_2$ be axis-parallel cubes in $\R^d$, of sidelengths $\delta_{Q_1} = \delta_1$ and $\delta_{Q_2} = \delta_2$, satisfying $ \delta_1 \delta_2 \geq 16 $. Consider the spatio-spectral limiting operator operator $T_{Q_1,Q_2} =  B_{Q_2} P_{Q_1} B_{Q_2}$ on $L^2(\R^d)$. Then for every $\epsilon \in (0,1/2)$,
\begin{equation}
\label{M_eps1:eqn}
M_\epsilon(T_{Q_1,Q_2})  \lesssim \log(\delta_1 \delta_2)^d \log(\epsilon^{-1})^d +  (\delta_1 \delta_2)^{d-1} \log( \delta_1 \delta_2) \log(\epsilon^{-1}).
\end{equation}
\end{lemma} 
\begin{proof}
An affine change of variables in the spatial and frequency variables of the form $(x,\xi) \mapsto (A x + x_0, A^{-1} \xi + \xi_0)$ ($A>0$) preserves the eigenvalues of a spatio-spectral limiting operator. Thus, we may reduce to the case  $Q_1 = [ -r,r]^d$, and $Q_2 = [0,1]^d$ for some $r \geq 8$, where $\delta_1 \delta_2 = 2r$. By Theorem 1.2 of \cite{ArieAzita23},
\begin{align*}
 M_\epsilon(T_{Q_1,Q_2}) 
&\lesssim \max \{ r^{d-1} \log(r) \log(\epsilon^{-1}), (\log(r) \log(\epsilon^{-1}))^d\} \\
&\sim  (\log(r) \log(\epsilon^{-1}))^d + r^{d-1} \log(r) \log(\epsilon^{-1}).
\end{align*}
This completes the proof of the lemma. 
\end{proof}

%
%

\subsection{Synthesis of component operators}

\begin{lemma}\label{decomp:lem} 
Let $S$ be a measurable set in $\R^d$, and let $F = \bigcup_{Q \in \cQ} Q$, where $\cQ$ is a finite collection of closed cubes in $\R^d$ with pairwise disjoint interiors. Suppose there exist functions $A(t), B(t) \geq 0$ such that 
\begin{equation}\label{component_plunge_assump_new}
\begin{aligned}
M_\epsilon(T_{Q,S})  \leq A(\delta_Q) \log(\epsilon^{-1})^d + B(\delta_Q) \log(\epsilon^{-1}) \mbox{ for all } Q \in \cQ, \; \epsilon \in (0,1/2).
\end{aligned}
\end{equation}
Then the operator $T_{F,S} = \sum_{Q \in \cQ} T_{Q,S}$ satisfies
\[
M_\epsilon(T_{F,S}) 
\leq 
d! 2^{d+2} A \log(\epsilon^{-1})^d + 8B \log(\epsilon^{-1}) \mbox{ for all } \epsilon \in (0,1/2),
\]
where $A = \sum_{Q \in \cQ} A(\delta_Q)$ and $B= \sum_{Q \in \cQ} B(\delta_Q)$.
\end{lemma}

\begin{proof}
Recall the Hankel operator is defined by $H_{K,S} = (1-B_S) P_K B_S$ for measurable sets $K,S \subset \R^d$. Observe that  $H_{F,S} = \sum_{Q \in \cQ} H_{Q,S}$, since $F = \bigcup_{Q \in \cQ} Q$ and the cubes in $\cQ$ have pairwise disjoint interiors.

For $\epsilon \in (0,1/2)$, and $0 < p \leq 1/2$, we apply \eqref{cheb:ineq}, the  identity \eqref{op_fact2:eqn}, the subadditivity of $(2p)^{\text{th}}$ powers of Schatten $(2p)$-norms, and \eqref{op_fact2:eqn} again, and obtain:
\begin{equation}\label{bound-x}
\begin{aligned}
M_\epsilon(T_{F,S}) \leq 2^p \epsilon^{-p} \| T_{F,S} - T_{F,S}^2 \|_p^p  = 2^p \epsilon^{-p} \| H_{F,S} \|_{2p}^{2p} &\leq 2 \epsilon^{-p}  \sum_{Q \in \cQ} \| H_{Q,S} \|_{2p}^{2p} \\
&= 2 \epsilon^{-p} \sum_{Q \in \cQ} \| T_{Q,S} - T_{Q,S}^2 \|_{p}^{p}.
\end{aligned}
\end{equation}
We use assumption \eqref{component_plunge_assump_new} to control the individual terms in the sum from \eqref{bound-x}. Indeed, by Lemma \ref{lem1},
\[
\| T_{Q,S} - T_{Q,S}^2 \|_{p}^p 
\leq 
d! A(\delta_Q) p^{-d} + B(\delta_Q)  p^{-1} \mbox{ for any } Q \in \cQ.
\]
We apply this estimate in the bound \eqref{bound-x}, and obtain
\[
M_\epsilon(T_{F,S}) 
\leq 
2\epsilon^{-p} \sum_{Q \in \cQ} \left[ d! A(\delta_Q) p^{-d} +  B(\delta_Q) p^{-1} \right]. 
\]
Now, let $p := \frac{1}{2\log(\epsilon^{-1})} \in (0,1/2]$. Then $\epsilon^{-p} = 2^{ \log(\epsilon^{-1}) p} = \sqrt{2}$, $p^{-1} = 2\log(\epsilon^{-1})$ and $p^{-d} = 2^d \log(\epsilon^{-1})^d$. Therefore, the preceding estimate implies the  conclusion of the lemma. 
\end{proof}

\section{Spatial  domain decomposition}\label{sec:dd1}

In this section we prove an intermediate result corresponding to a version of Theorem \ref{thm1} when the frequency domain $S$ is a cube.

\begin{theorem}\label{stageIdecomp:thm}
Let $d \geq 2$ and let $F \subset \R^d$ be a bounded measurable set having a maximally Ahlfors regular boundary with regularity constant $\kappa_{\partial F} >0$ (see \eqref{Ahlfors_cond} for the definition).  Let $ T_{F} 
:= T_{F,Q_0} = B_{Q_0} P_{F} B_{Q_0}$ be a spatio-spectral limiting operator associated to the region $  F \times Q_0 \subset \R^d \times \R^d$, where $Q_0 \subset \R^d$ is a cube satisfying $\delta_{Q_0}^{d-1} \cH_{d-1}( \partial F ) \geq (32)^{d-1}$. Then for any $\epsilon \in (0,1/2)$,
\[
M_\epsilon(T_F) \lesssim  \delta_{Q_0}^{d-1} \frac{\cH_{d-1}(\partial F)}{\kappa_{\partial F}}  \left[ \log(\epsilon^{-1})^d +  \log (\delta_{Q_0}^{d-1} \cH_{d-1}(\partial F ))^2 \log(\epsilon^{-1}) \right].
\]
\end{theorem}

Notice that  $ \cH_{d-1}(\partial Q_0) \sim  \delta_{Q_0}^{d-1}$ and $\kappa_{\partial Q_0}\sim 1$, so  the bound in Theorem \ref{stageIdecomp:thm} is similar to (in fact, slightly stronger than) the bound in our main result, Theorem \ref{thm1}, for $S = Q_0$. 

Recall that the regularity constant of a domain boundary is scaling-invariant, and the eigenvalues of the operator $T_{F,Q_0}$ are invariant under the joint rescalings $Q_0 \mapsto t^{-1} Q_0$ and $F \mapsto t F$. Therefore, the inequality in the conclusion of the theorem is scaling-invariant. Thus, we may assume without loss of generality that $\delta_{Q_0} = 1$ and   $\cH_{d-1}(\partial F) \geq (32)^{d-1}$. Hence, $\cH_{d-1}(\partial F)^{1/(d-1)} \geq 32$. 

Our approach will be to first verify the conclusion of Theorem~\ref{stageIdecomp:thm} for certain inner and outer approximations of \(F\) arising from Proposition~\ref{whitney_truncated:prop}. 
We introduce these approximations now and state this result in Lemma \ref{lem6} below.

We employ a dyadic approximation as in Section \ref{dyadic approximation of domains}, and apply Proposition \ref{whitney_truncated:prop} to the domain $F$. For this, consider the family of cubes $\cW = \cW(\eta)$ with $\eta = 16$. Note that $\eta \leq \frac{1}{2} \cH_{d-1}(\partial F)^{1/(d-1)}$ as required in the proposition's hypothesis. We denote this family as $\cW_+$.
Then $\cW_+$ is a cover of $F$ by cubes with pairwise disjoint interiors and sidelengths that are integer powers of $2$, such that 
\begin{equation}\label{W_+:prop1}
16 \leq \delta_{Q_F} \leq \cH_{d-1}(\partial F)^{1/(d-1)} \quad \mbox{for } Q_F \in \cW_+,
\end{equation}
and for each integer $k$,
\begin{equation}
\label{cubecount1}
\# \{ Q_F \in \cW_+ : \delta_{Q_F} = 2^k \} \leq C_1 \frac{\cH_{d-1}( \partial F)}{\kappa_{\partial F}} 2^{k(1-d)}.
\end{equation}
We also define $\cW_- = \{ Q_F \in \cW_+ : \delta_{Q_F} \geq 32\}$.

Because $\cW_+$ is a cover of $F$, 
\[
F_+ := \bigcup_{Q_F \in \cW_+} Q_F \supset F.
\]
Condition 2 of Proposition \ref{whitney_truncated:prop} asserts that $Q_F \subset F$ for all $Q_F \in \cW_-$, thus 
\[
F_- := \bigcup_{Q_F \in \cW_-} Q_F \subset F.
\]
Here, \(F_+\) is the outer approximation of \(F\), and \(F_-\) is the inner approximation of \(F\).

Let $$T_{F_+} := T_{F_+, Q_0} = \sum_{Q_F \in \cW_+} T_{Q_F,Q_0}, \quad \text{and,} \quad T_{F_-} := T_{F_-,Q_0} = \sum_{Q_F \in \cW_-} T_{Q_F,Q_0}.$$ 

\begin{lemma}\label{lem6} 
For any $\epsilon \in (0,1/2)$, 
\[
M_\epsilon(T_{F_\pm}) \lesssim \frac{\cH_{d-1}( \partial F)}{\kappa_{\partial F}} \left[ \log(\epsilon^{-1})^d +  \log (\cH_{d-1}( \partial F))^2 \log(\epsilon^{-1}) \right].
\]
\end{lemma} 
\begin{proof}
We apply Lemma \ref{lem3} to estimate the components $T_{Q_F,Q_0}$ for $Q_F \in \cW_+$. Note that $\delta_{Q_F} \delta_{Q_0} = \delta_{Q_F} \geq 16$ for $Q_F \in \cW_+$. Thus, we obtain
\begin{equation}\label{comp_bound1:eqn}
M_\epsilon(T_{Q_F,Q_0}) 
\lesssim \log(\delta_{Q_F})^d \log(\epsilon^{-1})^d + \delta_{Q_F}^{d-1} \log(\delta_{Q_F}) \log(\epsilon^{-1}) \mbox{ for all } Q_F \in \cW_+, \; \epsilon \in (0,1/2).
\end{equation}

For $k \geq 4$, let $\cW_k$ be the set of all $Q_F \in \cW_+$ of sidelength $\delta_{Q_F}= 2^k$, and put
\begin{equation}\label{kmax_defn}
k_{\max} := \lfloor \log (\cH_{d-1}(\partial F)^{1/(d-1)}) \rfloor.
\end{equation}
Note that $d \geq 2$, and thus, $k_{\max} \leq \log (\cH_{d-1}(\partial F))$. 
According to \eqref{W_+:prop1}, $\cW_+ = \bigcup_{4 \leq k \leq k_{\max}} \cW_k$. By definition of $\cW_-$, also $\cW_- = \bigcup_{5 \leq k \leq k_{\max}} \cW_k$. 

We apply Lemma \ref{decomp:lem} with $(S,F,\cQ)$ in the context of the lemma taken to be $(Q_0,F^+,\cW^+)$ here. Thus, we will obtain an estimate on the eigenvalues of $T_{F_+} = T_{F_+,Q_0}$.
The main condition \eqref{component_plunge_assump_new} of Lemma \ref{decomp:lem} is satisfied  with $A(\delta) = \log(\delta)^d$ and $B(\delta) = \delta^{d-1} \log(\delta)$; see \eqref{comp_bound1:eqn}. It is required to estimate $A:= \sum_{Q_F \in \cW_+} A(\delta_{Q_F})$ and $B := \sum_{Q_F \in \cW_+} B(\delta_{Q_F})$. Write $A = \sum_{k=4}^{k_{\max}} \# \cW_k \cdot A(2^k)$, and use the bound $\# \cW_k \leq C_1 \cH_{d-1}(\partial F) \kappa_{\partial F}^{-1} 2^{k(1-d)}$ given in \eqref{cubecount1}, to obtain
\[
A = \sum_{k=4}^{k_{\max}} \log(2^k)^d \#\cW_k 
\leq C_1
\cH_{d-1}(\partial F) \kappa_{\partial F}^{-1} \cdot \sum_{k=4}^{k_{\max}} k^d 2^{k(1-d)}.
\]
For $d \geq 2$,  $\sum_{k=4}^\infty k^d 2^{k(1-d)} = C< \infty$, where $C$ is bounded by a constant independent of $d$.
In fact, it is not difficult to show that $C \leq 16$. So, $A \lesssim \cH_{d-1}(\partial F) \kappa_{\partial F}^{-1}$.

Similarly, using the fact that $\sum_{k=4}^{k_{\max}} k \leq k_{\max}^2 \leq \log(\cH_{d-1}(\partial F))^2$,
\[
\begin{aligned}
B = \sum_{k=4}^{k_{\max}} (2^k)^{d-1} \log(2^k) \# \cW_k 
&\leq C_1
\cH_{d-1}(\partial F) \kappa_{\partial F}^{-1} \cdot \sum_{k=4}^{k_{\max}} k \\
&
\lesssim \cH_{d-1}(\partial F) \kappa_{\partial F}^{-1} \log (\cH_{d-1}( \partial F))^2.
\end{aligned}
\]
Therefore, from Lemma \ref{decomp:lem} we deduce that $M_\epsilon(T_{F_+}) \lesssim  A\log(\epsilon^{-1})^d + B\log(\epsilon^{-1})$ for all $\epsilon \in (0,1/2)$, as required.

Since $\cW_-$ is a subcollection of $\cW_+$ we can apply Lemma \ref{decomp:lem} for $\cQ = \cW_-$ in the same manner as above, and deduce that the corresponding estimate holds for $T_{F_-}$. 
\end{proof}

We can determine the approximate location of the eigenvalue of the operator $ T_{F_+}$ or $T_{F_-}$ crossing $\lambda=1/2$ using the following result, which is taken verbatim from \cite{MaRoSp23}. 

\begin{lemma}[Lemma 5.3 of \cite{MaRoSp23}]\label{lem2} 
 
For any trace class operator $0 \leq T \leq 1$, 
\begin{enumerate}
    \item $\lambda_n(T) \leq \frac{1}{2}$ for every $n \geq \lceil \tr(T) \rceil + \max \{ 2 \tr(T-T^2),1\}$
    \item $\lambda_n(T) \geq \frac{1}{2}$ for every $n \leq \lceil \tr(T) \rceil - \max \{ 2 \tr(T-T^2),1\}$.
\end{enumerate}
\end{lemma}

To continue the proof of Theorem \ref{stageIdecomp:thm}  we need to estimate $\tr(T_{F_\pm})$ and $\tr(T_{F_\pm} - T_{F_{\pm}}^2)$. Using Proposition \ref{prop:SSLO_props}, we have $\tr(T_{F_\pm}) = (2 \pi)^{-d} | F_\pm| |Q_0|  = (2 \pi)^{-d} | F_\pm|$.
By \eqref{cubecount1},
\begin{equation}\label{trace_diff_est:eqn}
\tr(T_{F_+}) - \tr(T_{F_-}) = (2 \pi)^{-d} (| F_+| - |F_-|) = (2 \pi)^{-d} \sum_{Q_F \in \cW_+, \delta_{Q_F} = 16} |Q_F|  \lesssim \frac{\cH_{d-1}( \partial F)}{\kappa_{\partial F}}.
\end{equation}
Combining Lemma \ref{lem6} and Lemma \ref{lem1} for $p=1$, we obtain
\begin{equation}\label{trace_est2:eqn}
\tr(T_{F_\pm} - T_{F_\pm}^2) = \| T_{F_\pm} - T_{F_\pm}^2\|_1 \lesssim \frac{\cH_{d-1}( \partial F)}{\kappa_{\partial F}} \log(\cH_{d-1}( \partial F))^2.
\end{equation}

We finally estimate the function $M_\epsilon(T_F)$. Because $F_- \subset F \subset F_+$, we have the inequality of operators
\begin{equation}\label{ineq_op1:eqn}
T_{F_-} \leq T_F \leq T_{F_+}.
\end{equation}
From \eqref{ineq_op1:eqn}, $\lambda_n(T_{F_-})\leq \lambda_n(T_F)\leq \lambda_n(T_{F_+})$ for all $n$. Therefore, for any $\epsilon \in (0,1/2)$, we can write, 
\begin{align*} 
M_\epsilon(T_F) 
&= 
\# \{ n :  \lambda_n(T_F) \in (\epsilon,1-\epsilon) \} 
\\&\leq 
\#\{n : 1/2 \leq \lambda_n(T_{F_-}) < 1 - \epsilon\}  + \#\{n : \epsilon < \lambda_n(T_{F_+}) \leq  1/2 \} + \#\{n : \lambda_n(T_{F_-})< 1/2  < \lambda_n(T_{F_+})\} \\
& =: E_1 + E_2 + E_3.
\end{align*}
Note that $E_1 \leq \# \{n :  \lambda_n(T_{F_-}) \in (\epsilon,1-\epsilon) \}$  and 
$E_2 \leq \# \{ n :  \lambda_n(T_{F_+}) \in (\epsilon,1-\epsilon) \}$. Thus, by  Lemma \ref{lem6}, 
\[E_1 + E_2  \lesssim  \frac{\cH_{d-1}( \partial F)}{\kappa_{\partial F}} \left[ \log(\epsilon^{-1})^d +  \log (\cH_{d-1}( \partial F))^2 \log(\epsilon^{-1}) \right].
\]
We apply Lemma \ref{lem2} to estimate the term $E_3$: 
\begin{align*}
       E_3 &\leq \#\{n:  \lceil \tr(T_{F_-}) \rceil - \max \{ 2 \tr(T_{F_-} - T_{F_-}^2),1\}  < n < 
       \lceil \tr(T_{F_+}) \rceil  + \max \{ 2 \tr(T_{F_+} - T_{F_+}^2),1\} \\ 
       &\lesssim     \lceil \tr(T_{F_+}) \rceil  -  \lceil \tr(T_{F_-}) \rceil  +  \max \{ 2 \tr(T_{F_+}-T_{F_+}^2),1\}  +  \max \{ 2 \tr(T_{F_-} - T_{F_-}^2),1\}.
\end{align*}
According to \eqref{trace_diff_est:eqn} and \eqref{trace_est2:eqn},
\[
E_3 \lesssim \frac{\cH_{d-1}( \partial F)}{\kappa_{\partial F}} \log(\cH_{d-1}( \partial F))^2 + 1 \lesssim \frac{\cH_{d-1}( \partial F)}{\kappa_{\partial F}} \log(\cH_{d-1}( \partial F))^2.
\]
Applying the bounds on $E_1,E_2$, and $E_3$,
\[
M_\epsilon(T_F)   \lesssim \frac{\cH_{d-1}( \partial F)}{\kappa_{\partial F}} \left[ \log(\epsilon^{-1})^d +  \log (\cH_{d-1}( \partial F))^2 \log(\epsilon^{-1}) \right].
\]
This completes the proof of Theorem \ref{stageIdecomp:thm}. (Recall $\delta_{Q_0} = 1$.)

\section{Frequency domain decomposition}
\label{sec:dd2}

In this section, we complete the proof of Theorem \ref{thm1}.  Fix bounded measurable sets $F,S \subset \R^d$ having maximally Ahlfors regular boundaries with regularity constants $\kappa_{\partial F}$ and $\kappa_{\partial S}$, respectively. We suppose that $\cH_{d-1}( \partial F) \cdot \cH_{d-1}(\partial S) \geq  64^{d-1}$.
By applying a joint rescaling of the form $F \mapsto t F$ and $S \mapsto t^{-1} S$, we may assume without loss of generality that the domains satisfy
\[
\cH_{d-1}(\partial F) = 2^{d-1}, 
\quad 
\cH_{d-1}(\partial S) \geq 32^{d-1}.
\]

The operators $T_{F,S}= B_{S} P_F B_{S}$  and  $T_{S, F}= B_{F} P_{S} B_{F}$ have identical eigenvalue sequences -- see Proposition \ref{prop:SSLO_props}. So, it will be sufficient to estimate the eigenvalues of the second operator. 


We apply Proposition \ref{whitney_truncated:prop} to the set $S$ and let $\cW_+$ be the resulting family $ \cW = \cW(\eta)$ for $\eta = 16$. (Observe that $\eta \leq \frac{1}{2} \cH_{d-1}(\partial S)^{1/(d-1)}$.) Accordingly, $\cW_+$ is a family of cubes having pairwise disjoint interiors and sidelengths that are integer powers of $2$ -- furthermore, 
\begin{equation}\label{sl_bd}
16 \leq \delta_{Q_S} \leq \cH_{d-1}(\partial S)^{1/(d-1)} \mbox{ for all } Q_S \in \cW_+.
\end{equation}
We set $\cW_- := \{ Q_S \in \cW_+ : \delta_{Q_S} \geq 32\}$. According to Proposition \ref{whitney_truncated:prop},  $\cW_+$ is a cover of $S$, the cubes in $\cW_-$ are contained in $S$, and for each integer $k$,
\begin{equation}
\label{cubecount1_new}
\# \{ Q_S \in \cW_+ : \delta_{Q_S} = 2^k \} \leq C_1 \frac{\cH_{d-1}(\partial S)}{\kappa_{\partial S} } 2^{k(1-d)}.
\end{equation}
Thus,
\[
S_{-} := \bigcup_{Q_S \in \cW_-} Q_S \subset S \subset S_+  := \bigcup_{Q_S \in \cW_+} Q_S
\]
Here, $S_{-}$ is an inner approximation of $S$ and $S_+$ is an outer approximation of $S$.

We set 
\[
T_{S_-} := T_{S_-,F} = \sum_{Q_S \in \cW_-} T_{Q_S,F} \mbox{ and } T_{S_+} := T_{S_+,F} = \sum_{Q_S \in \cW_+} T_{Q_S,F}.
\]

We shall apply Lemma \ref{decomp:lem} to estimate the distribution of the eigenvalues of $T_{S_-}$ and $T_{S_+}$. As a result, we obtain:

\begin{lemma} \label{lem7} For any $\epsilon \in (0,1/2)$,
\[
M_\epsilon(T_{S_\pm})  \lesssim  \frac{\cH_{d-1}(\partial S)}{ \kappa_{\partial F} \cdot \kappa_{\partial S}}  \cdot \biggl\{ \log( \cH_{d-1}(\partial S)) \log(\epsilon^{-1})^d + 
\log( \cH_{d-1}(\partial S))^3   \log(\epsilon^{-1}) \biggr\}.
\]
\end{lemma} 
\begin{proof}

For each $Q_S \in \cW_+$, we have $\delta_{Q_S}^{d-1} \cH_{d-1}( \partial F ) \geq 16^{d-1} 2^{d-1} = 32^{d-1}$. The operator $T_{Q_S,F} =  B_F P_{Q_S} B_F$ has the same eigenvalues as the operator $T_{F,Q_S} = B_{Q_S} P_F B_{Q_S}$, and thus, by Theorem \ref{stageIdecomp:thm}, we have
\begin{equation}\label{comp_bound1_new:eqn}
M_\epsilon(T_{Q_S,F}) \lesssim  
\delta_{Q_S}^{d-1} \kappa_{\partial F}^{-1} \left[ \log(\epsilon^{-1})^d +  \log(\delta_{Q_S})^2 \log(\epsilon^{-1}) \right] \;\; \mbox{ for all } Q_S \in \cW_+, \; \epsilon \in (0,1/2).
\end{equation}
(Recall we have normalized $\cH_{d-1} (\partial F) = 2^{d-1}$.)

Thanks to \eqref{sl_bd} we can write $ \cW_+ = \bigcup_{4 \leq k \leq k_{\max}} \cW_k$, where $\cW_k$ is the set of  cubes in $\cW_+$ of sidelength $2^k$, and
\begin{equation}
    \label{k_max:defn2}
    k_{\max} := \lfloor \log ( \cH_{d-1}(\partial S)^{1/(d-1)}) \rfloor.
\end{equation}
Observe that $k_{\max} \leq \log (\cH_{d-1}(\partial S))$. According to the definition of $\cW_-$, also $\cW_- = \bigcup_{5 \leq k \leq k_{\max}} \cW_k$.

Condition \eqref{component_plunge_assump_new} of Lemma \ref{decomp:lem} is satisfied for $\cQ = \cW_+$, when the frequency domain $S$ in Lemma \ref{decomp:lem} is taken to be the spatial domain $F$ here, with $A(\delta) = \delta^{d-1} \kappa_{\partial F}^{-1}$, and $B(\delta) = \delta^{d-1} \log(\delta)^2 \kappa_{\partial F}^{-1}$; see \eqref{comp_bound1_new:eqn}. It is required to estimate $A:= \sum_{Q_S \in \cW_+} A(\delta_{Q_S})$ and $B := \sum_{Q_S \in \cW_+} B(\delta_{Q_S})$. Write $A = \sum_{k=4}^{k_{\max}} \# \cW_k \cdot A(2^k)$, and use the bound $\#\cW_k \leq C_1 \cH_{d-1}(\partial S) \kappa_{\partial S}^{-1} 2^{k(1-d)}$ given in \eqref{cubecount1_new}, to obtain
\[
\begin{aligned}
A = \sum_{k=4}^{k_{\max}} (2^k)^{d-1} \kappa_{\partial F}^{-1} \cdot \#\cW_k &\leq C_1 \cH_{d-1}(\partial S) \kappa_{\partial S}^{-1} \cdot \kappa_{\partial F}^{-1} \cdot \sum_{k=4}^{k_{\max}} 1 \\
&\lesssim \cH_{d-1}(\partial S) \kappa_{\partial S}^{-1} \cdot \kappa_{\partial F}^{-1} \cdot \log(\cH_{d-1}(\partial S)).
\end{aligned}
\]
Similarly,
\[
\begin{aligned}
B = \sum_{k=4}^{k_{\max}} (2^k)^{d-1} \log(2^k)^2 \kappa_{\partial F}^{-1}  \cdot \# \cW_k &\leq C_1  \cH_{d-1}(\partial S) \kappa_{\partial S}^{-1} \cdot \kappa_{\partial F}^{-1} \cdot \sum_{k=4}^{k_{\max}} k^2 \\
& \lesssim \cH_{d-1}(\partial S) \kappa_{\partial S}^{-1} \cdot \kappa_{\partial F}^{-1} \cdot \log(\cH_{d-1}(\partial S))^3.
\end{aligned}
\]
Then, from Lemma \ref{decomp:lem} we deduce that $M_\epsilon(T_{S_+}) \lesssim  A \log(\epsilon^{-1})^d + B  \log(\epsilon^{-1})$ for all $\epsilon \in (0,1/2)$, as desired.

Similarly, we apply Lemma \ref{decomp:lem} for the family $\cQ = \cW_-$, to deduce that the corresponding estimate holds for $T_{S_-}$. 
\end{proof}

As in the previous section, we now prepare to apply Lemma \ref{lem2} to the operators $S = T_{S_\pm}$. We first derive the required estimates on $\tr(T_{S_\pm})$ and $\tr(T_{S_\pm} - T_{S_\pm}^2)$. 

First, we have
\begin{equation}\label{trace_est1_new:eqn}
    \tr(T_{S_\pm}) = (2 \pi)^{-d} | S_\pm| |F|
\end{equation}
By the isoperimetric inequality, $|F| \lesssim \cH_{d-1}(\partial F)^{d/(d-1)} \lesssim 1$. Therefore, by \eqref{cubecount1_new},
\begin{equation}\label{trace_diff_est_new:eqn}
\begin{aligned}
    \tr(T_{S_+}) - \tr(T_{S_-}) &=(2 \pi)^{-d} (| S_+| - |S_-|) | F| = (2 \pi)^{-d} \sum_{Q_S \in \cW_+, \delta_{Q_S} = 16} |Q_S| \cdot | F| \lesssim  \frac{\cH_{d-1}(\partial S)}{\kappa_{\partial S}}.
\end{aligned}
\end{equation}
Combining Lemma \ref{lem7} and Lemma \ref{lem1} for $p=1$, we obtain
\begin{equation}\label{trace_est2_new:eqn}
\tr(T_{S_\pm} - T_{S_\pm}^2) = \| T_{S_\pm} - T_{S_\pm}^2\|_1  \lesssim  \frac{\cH_{d-1}(\partial S)}{\kappa_{\partial S}}
\kappa_{\partial F}^{-1} \log(\cH_{d-1}(\partial S))^3.
\end{equation}

We abbreviate $T_{S} = T_{S,F}$. Because $S_- \subset S \subset S_+$, we have the inequality of operators
\[
T_{S_-} \leq T_{S} \leq T_{S_+}.
\]
Therefore, $ \lambda_n(T_{S_-}) \leq \lambda_n (T_{S}) \leq \lambda_n(T_{S_+})$ for all $n$. We replicate the argument at the end of the previous section -- specifically, we use  Lemma \ref{lem7}, and apply Lemma \ref{lem2} to the operators $S = T_{S_\pm}$ using the estimates \eqref{trace_diff_est_new:eqn} and \eqref{trace_est2_new:eqn} -- and obtain that
\[
M_\epsilon(T_{S}) \lesssim  \frac{\cH_{d-1}(\partial S)}{\kappa_{\partial S}} \kappa_{\partial F}^{-1} \cdot \biggl\{ \log( \cH_{d-1}(\partial S)) \log(\epsilon^{-1})^d  + \log (\cH_{d-1}(\partial S))^3 \log(\epsilon^{-1}) \biggr\} .
\]
Because the operators $T_{S} = T_{S,F} = B_F P_{S} B_F$ and $T_{F, S} = B_{S} P_F B_{S}$ have identical eigenvalue sequences, this completes the proof of Theorem \ref{thm1}. (Recall we have normalized $\cH_{d-1}(\partial F) = 2^{d-1}$.)

\section{Proof of Theorem \ref{thm2}}\label{sec:pfrem1}

We reproduce the argument in \cite{MaRoSp23} with slight modifications to account for the different forms of our estimates. Let $T := T_{F,S}$. By applying Lemma \ref{lem1} (for $p=1$) and Theorem \ref{thm1}, there exists a constant $C_d > 0$ such that
\[
\tr(T - T^2) \leq C_d \frac{ \cH_{d-1}(\partial F) \cH_{d-1}(\partial S)}{\kappa_{\partial F} \kappa_{\partial S}} \log(\cH_{d-1}(\partial F) \cH_{d-1}(\partial S))^3 =: H(F,S).  
\]
From Proposition \ref{prop:SSLO_props} and Lemma \ref{lem2}, we deduce that 
\begin{equation}\label{eqn:crossing_index}
\left\{
\begin{aligned}
n \geq \lceil (2 \pi)^{-d} |S| \cdot |F| \rceil + \max\{ 2 H(F,S), 1 \}  =: H_1(F,S) \implies \lambda_n(T) \leq \frac{1}{2},\\
n \leq \lceil (2 \pi)^{-d} |S| \cdot |F| \rceil - \max\{ 2 H(F,S), 1 \} =: H_2(F,S) \implies \lambda_n(T) \geq \frac{1}{2}.
\end{aligned}
\right.
\end{equation}
For $\epsilon \in (0,1)$, define $\epsilon_0 := \min \{\epsilon, 1-\epsilon\} \leq 1/2$, and let $0 < \tau < \epsilon_0$. Then, for $Z_\tau(T) := \{ n \in \N : \lambda_n(T) \in (\tau,1-\tau)\}$, \eqref{eqn:crossing_index} gives that
\[
\{1,\cdots, \lfloor H_1(F,S) \rfloor \} \setminus Z_\tau(T) \subset \{n \in \N : \lambda_n(T) > \epsilon \} \subset \{1,\cdots, \lceil H_2(F,S) \rceil \} \cup Z_\tau(T),
\]
where by convention $\{1,\cdots, \lceil H_2(F,S) \rceil \} = \emptyset$ if $H_2(F,S) < 0$. Consequently,
\[
H_2(F,S) - 1 - \#( Z_\tau(T)) \leq \# \{n \in \N : \lambda_n(T) > \epsilon \} \leq H_1(F,S) + 1 + \#(Z_\tau(T)).
\]
By definition of the eigenvalue counting function, $M_\tau(T) =\#(Z_\tau(T))$. Thus, by rearranging the previous expression, we can apply Theorem \ref{thm1} to give that
\[
\begin{aligned}
\left| \# \{n \in \N : \lambda_n(T) > \epsilon \} - (2\pi)^{-d} |F| \cdot |S| \right| & \lesssim \frac{\cH_{d-1}(\partial F)}{\kappa_{\partial F}} \frac{\cH_{d-1}(\partial S)}{\kappa_{\partial S}}  \biggl\{  \log\left(\cH_{d-1}(\partial F) \cdot \cH_{d-1}(\partial S)\right)  \log(\tau^{-1})^d \\
& + \log \left(\cH_{d-1}(\partial F) \cdot \cH_{d-1}(\partial S)\right)^3 \log(\tau^{-1}) \biggr\}.
\end{aligned}
\]
Letting $\tau \rightarrow \epsilon_0$ yields \eqref{eqn:main2}.

\addcontentsline{toc}{section}{References}

\printbibliography

\end{document}